\documentclass[12pt, draft]{amsart}
\usepackage[english]{babel}
\usepackage{latexsym}

\usepackage{amsmath}

\newcommand{\ignore}[1]{}

\newcommand{\hide}[1]{}

\reversemarginpar

\DeclareMathOperator{\ad}{ad}

\DeclareMathOperator{\Der}{Der}

\newcommand{\F}{\mathbb F}

\newcommand{\Z}[0]{\mathbb Z}

\newtheorem{dummy}{Dummy}

\numberwithin{dummy}{section}
\numberwithin{equation}{section}

\newtheorem{theorem}[dummy]{Theorem}

\theoremstyle{definition}
\newtheorem{definition}[dummy]{Definition}

\theoremstyle{remark}
\newtheorem{rem}[dummy]{Remark}
\newtheorem*{rem*}{Remark to ourselves}

\begin{document}

\bibliographystyle{amsalpha}

\author{Marina Avitabile}
\email{marina.avitabile@unimib.it}
\address{Dipartimento di Matematica e Applicazioni\\
  Universit\`a degli Studi di Milano - Bicocca\\
  via Cozzi 53\\
  I-20125 Milano\\
  Italy}

\author{Sandro Mattarei}
\email{mattarei@science.unitn.it}
\address{Dipartimento di Matematica\\
  Universit\`a degli Studi di Trento\\
  via Sommarive 14\\
  I-38050 Povo (Trento)\\
  Italy}

\title[Nottingham {L}ie algebras]{Nottingham {L}ie algebras with diamonds \\of finite and infinite type}

\date{\today}
\begin{abstract}
We consider a class of infinite-dimensional, modular, graded Lie algebras,
which includes the graded Lie algebra associated to the Nottingham group with respect to its lower central series. We
identify two subclasses of {\em Nottingham Lie algebras} as loop algebras of finite-dimensional simple Lie algebras of Hamiltonian Cartan type.
A property of Laguerre polynomials of derivations, which is related to toral switching, plays a crucial role
in our constructions.
\end{abstract}
\maketitle

\section{Introduction}

{\em Nottingham Lie algebras} owe their name to one remarkable special case:
the graded Lie ring with the lower central series of the
Nottingham group, over the field of $p$ elements, with $p>2$~\cite{Jenn,John,Cam}.
Although the Nottingham group over $\F_p$
is a very complex object,
its lower central series descends in a simple and regular way,
the lower central quotients being generally one-dimensional,
except for one quotient every $p-1$ being two-dimensional.
The essential feature which leads to the generalization considered here is that the Nottingham group is a pro-$p$ group
of {\em width two} and {\em obliquity zero}, in the language
of~\cite{KL-GP}.
Pro-$p$ groups with those two properties were called {\em thin} in~\cite{CMNS},
and the properties carry over naturally to graded Lie algebras.

The following equivalent but more practical definition is
available for Lie algebras.
A {\em thin Lie algebra} is a graded Lie algebra $L=\bigoplus_{k=1}^{\infty}L_{k}$ over a field $\F$,
with $\dim(L_1)=2$ and satisfying the {\em covering property}
\begin{equation}\label{eq:covering}
L_{i+1}=[u, L_{1}] \quad \textrm{for all $0 \neq u\in L_{i}$, for all $i\geq 1$}.
\end{equation}
For simplicity in this paper we supplement this definition with the
assumption that $L$ has infinite dimension.
It follows that a thin Lie algebra $L$ has trivial centre, is generated by
$L_1$, and that homogeneous components can only have dimension one or two.
Those of dimension two are called {\em diamond}
(for reasons relating to the lattice of open normal subgroups in the
pro-$p$ group case, see~\cite{CMNS}),
but there are good reasons to grant the status of {\em fake diamonds}
to certain one-dimensional components.

Diamonds are numbered in the order of occurrence (including the fake ones, properly recognized), starting with the {\em first} diamond $L_1$.
It may happen that $L_1$ is the only diamond, but then $L$ belongs to the distinguished family of
{\em graded Lie algebra of maximal class,} which were introduced
in~\cite{CMN}.
Because those Lie algebras were completely classified
in~\cite{CN,Ju:maximal} (in the infinite-dimensional case),
we conveniently exclude them from the definition of thin Lie
algebras.
Then the degree $k$ where the second diamond occurs becomes the main parameter of a thin Lie algebra.
According to~\cite{AviJur}
(but see the simpler proof in~\cite[Section~3]{AviMat:A-Z},
where an updated definition accommodates for the peculiarities of characteristic two),
the parameter $k$ can only take the values $3$, $5$,
$q$ or $2q-1$, for some power $q$ of the characteristic $p$ in
case this is positive.
Each of the last two cases comprises a large variety of thin Lie
algebras, whose features are best described separately.
The introduction of~\cite{CaMa:Hamiltonian} contains a detailed exposition of the thin Lie algebras
with second diamond in degree $2q-1$.
The thin Lie algebras with second diamond in degree $q$ were named
{\em Nottingham Lie algebras} in~\cite{CaMa:Nottingham}.

According to~\cite{CaMa:Nottingham}, each diamond past the first in a Nottingham Lie algebra is assigned a {\em type,}
taking values in the underlying field augmented with infinity,
in a unique way which we recall in Section~\ref{sec:nott}.
The second diamond of a Nottingham Lie algebra always has type $-1$.
Diamonds of type zero or one are really one-dimensional
components, but it is convenient to allow them in certain
degrees and dub them {\em fake}.
Fake diamonds are recognized among the one-dimensional components
from certain relations which hold in them.
Because $1\equiv -1 \bmod 2$, in the special case of characteristic two the second diamond of a Nottingham
Lie algebra would be fake, and so one needs some care
to recognize a Nottingham Lie algebra as such.
This and more peculiarities of Nottingham Lie algebras of characteristic two
are discussed extensively in~\cite{AviMat:A-Z}, but in the present paper
we conveniently assume $p>2$.
We will, however, add separate remarks on how our results need to be
modified in characteristic two.

Various diamond patterns are possible, and we describe them in Section~\ref{sec:nott}.
There are Nottingham Lie algebras with all diamonds of infinite type, and others
with all diamonds of finite type.
The goal of this paper is a construction (whence an existence proof) for
Nottingham Lie algebras having diamonds of both finite and infinite types.
They have diamonds in all degrees congruent to $1$ modulo $q-1$,
but only one diamond every $p^s$ diamonds has finite type (starting with the second diamond, of necessity).
Furthermore, the finite diamond types follow an arithmetic progression, giving the algebras a periodic structure.
A distinction arises according to
whether this arithmetic progression is entirely contained in the prime field, or not.
Fake diamonds only occur in the former case.
This is actually the easier case, and its special case
where the arithmetic progression is the constant sequence $-1$ was already considered in~\cite{AviMat:A-Z}.
 
We state and prove our main results in Sections~\ref{sec:prime-field} and~\ref{sec:big-field}, respectively.

As for other examples of thin Lie algebras produced so far (including those with second diamond in degree $2q-1$), those with a periodic structure
can be obtained through a {\em loop algebra} construction, starting from a finite-dimensional
Lie algebra with a suitable cyclic grading.
In most cases those finite-dimensional Lie algebras are simple (or close to simple)
Lie algebras of the Cartan type $H$ (Hamiltonian).
In this paper we need two types of simple Hamiltonian Lie algebras, of dimension a power of $p$ and two less than a power of $p$,
whose definitions we recall in Section~\ref{sec:Cartan}.

While the correct Lie algebra to employ is generally easy to guess for dimension reasons
(where the absence of fake diamonds, or the presence of one or two in each period, calls for a Lie algebra of dimension a power of $p$, or one or two less),
explicitly producing a suitable cyclic grading can be a real challenge without the proper tool.
In lucky situations one can obtain the required cyclic gradings starting from the natural gradings
of those Hamiltonian Lie algebras.
In other cases one needs to pass to new gradings by what we may call a {\em grading switching}.
This procedure is related to {\em toral switching,} a fundamental tool in the theory of modular Lie algebras,
but it needs to be more general in one respect, because the cyclic grading of interest may not be
associated to a torus in any way.

Like toral switching, grading switching is based on taking some version of an exponential of a derivation
of the Lie algebra.
A toral switching based on {\em Artin-Hasse exponentials,} which makes sense for nilpotent derivations $D$
and reduces to {\em truncated exponentials} $\sum_{i=0}^{p-1}D^i/i!$ when $D^p=0$,
was described in~\cite{Mat:Artin-Hasse}.
In the present paper this would only allow one to deal with the case where the arithmetic progression of finite diamond types
is contained in the prime field.
In the other case we need a special instance of a completely general version of grading switching,
valid for arbitrary derivations, which is developed in~\cite{AviMat:Laguerre}.
We provide the required details in Section~\ref{sec:Laguerre}.

In his master's thesis~\cite{Sca:thesis}, written under the direction of the first author,
Claudio Scarbolo gave a construction for our Nottingham Lie algebras of Section~\ref{sec:big-field}
in the case of characteristic two,
based on direct calculations and taking advantage of
certain peculiarities of the characteristic
(see Remark~\ref{rem:char2-big}).

\section{Nottingham Lie algebras}\label{sec:nott}

We summarize here the more complete discussion of Nottingham Lie algebras given in~\cite[Section~2]{AviMat:A-Z},
restricting ourselves to information which is essential to our present goals.
Thus, suppose $L=\bigoplus_{i=1}^{\infty} L_{i}$ is a thin Lie algebra with second diamond in degree $q>3$,
a power of the odd characteristic $p$.
As anticipated in the Introduction, we generally call such $L$ a Nottingham Lie
algebra (but see~\cite[Remark~3.2]{AviMat:A-Z} for a motivated exclusion of one case with $q=5$).

Choose a nonzero element $Y\in L_1$ with $[L_2,Y]=0$.
According to~\cite{CaJu:quotients} we have
\begin{equation}\label{eq:first_chain}
C_{L_1}(L_2)=\cdots=C_{L_1}(L_{q-2})=\langle Y \rangle,
\end{equation}
where
$C_{L_1}(L_i)= \{ a \in L_1 \mid [a,b]=0 \textrm{ for every $b \in L_i$} \}$,
the centralizer of $L_i$ in $L_1$.
It was shown in~\cite{Car:Nottingham} (but see the more accessible proof given in~\cite[Section~2]{AviMat:A-Z}), that one can choose
$X\in L_1\setminus\langle Y\rangle$ such that
\begin{equation}\label{eq:second_diamond_type}
[V,X,X]=0=[V,Y,Y], \quad [V,Y,X]=-2[V,X,Y],
\end{equation}
where $V$ is any nonzero element of $L_{q-1}$.
Here we adopt the left-normed convention for iterated Lie brackets, namely, $[a,b,c]$ denotes $[[a,b],c]$.
 
An element $X$ satisfying relations~\eqref{eq:second_diamond_type}
is determined up to a scalar multiple, and so is $Y$,
which is any nonzero element of $C_{L_1}(L_2)$.
Relations~\eqref{eq:second_diamond_type} also imply that $L_{q+1}$, the homogeneous component
immediately following the second diamond, has dimension one.
 
This extends to a more general fact about thin Lie algebras
(not necessarily with second diamond in degree $q$):
in a thin Lie algebra of arbitrary characteristic with $\dim(L_3)=1$,
two consecutive components cannot both be diamonds.

Suppose now that $L_{i}$ is a diamond of $L$ in degree $i>1$.
Let $V$ be a non-zero element in $L_{i-1}$, which has dimension one according to the previous paragraph,
whence $[V,X]$ and $[V,Y]$ span $L_i$.
If the relations
\begin{equation}\label{eq:diamond_type}
[V,Y,Y]=0=[V,X,X], \quad
(1- \mu)[V,X,Y]=\mu [V,Y,X],
\end{equation}
hold for some $\mu \in \F$, then we say that the diamond $L_i$ has (finite) type  $\mu$.
In particular, relations~\eqref{eq:second_diamond_type} postulate that the second diamond $L_q$ has type $-1$.
The third relation in Equation~\eqref{eq:diamond_type} has a more natural appearance when written in terms of the generators
$Z=X+Y$ and $Y$, which were more convenient in~\cite{CaMa:Nottingham}, namely,
$[V,Z,Y]=\mu [V,Z,Z]$.
We also define $L_i$ to be a diamond of
type $\infty$ by interpreting the third relation in Equation~\eqref{eq:diamond_type} as $-[V,X,Y]=[V,Y,X]$ in this case.
We stress that the type of a diamond of a Nottingham Lie algebra $L$ is
independent of the choice of the graded generators $X$ and $Y$.
In fact, we have seen above that our requirement $[L_2,Y]=0$ together with relations~\eqref{eq:second_diamond_type}
determine $X$ and $Y$ up to scalar multiples,
and replacing them with scalar multiples does not affect relations~\eqref{eq:diamond_type}.

According to the defining relations~\eqref{eq:diamond_type}, a diamond of type $\mu=0$ should satisfy $[V,X,Y]=0=[V,X,X]$, while
a diamond of type $\mu=1$ should satisfy $[V,Y,X]=0=[V,Y,Y]$.
But then $[V,X]$, or $[V,Y]$, respectively, would be a central element in $L$ and,
therefore, should vanish because of the covering property~\eqref{eq:covering}.
Hence, $L_i$ would be one-dimensional after all, and hence not a genuine diamond.
Thus, strictly speaking, diamonds of type $0$ or $1$ cannot occur.
Nevertheless, we allow ourselves to call {\em fake diamonds} certain one-dimensional homogeneous components of $L$
and assign them type $0$ or $1$ if the corresponding relations hold, as given in Equation~\eqref{eq:diamond_type}.
This should be regarded as a convenient piece of terminology rather than a formal definition:
whenever a one-dimensional component $L_i$ formally satisfies the relations of a fake diamond of type $1$,
the next component $L_{i+1}$ satisfies those of a fake diamond of type $0$.
However, as a rule, only one of them ought to be called a fake diamond,
usually because it fits into a sequence of (genuine) diamonds occurring at regular distances.

Before continuing with our description of basic properties of Nottingham Lie
algebras we briefly comment on what happens if we temporarily relax our blanket assumptions $p$ odd and $q>3$.
Many thin Lie algebras $L$ exist where $L_3$ is the second diamond, and a summary of the known ones is
given in~\cite[Remark~3.4]{AviMat:A-Z}.
In particular, in characteristic $p=3$ they
include Lie algebras which ought to be included in the class of
Nottingham Lie algebras, see~\cite[Remark~3.4]{AviMat:A-Z}.
This is so when generators $X$ and $Y$ can be chosen such that
relations~\eqref{eq:second_diamond_type} hold,
even though uniqueness of $\langle X\rangle$ and $\langle Y\rangle$ fails in this case.
With this extended definition the Nottingham Lie algebras which we construct in Sections~\ref{sec:big-field} and~\ref{sec:prime-field}
will be such also for $q=3$, with the diamond types determined by relations~\eqref{eq:diamond_type} as usual.
The picture for characteristic two is more peculiar, because the second
diamond fails to reveal itself as such, being fake of type $1=-1$.
It can even occur that all diamonds past the first are fake, and
such a Nottingham Lie algebra disguises itself as a graded Lie
algebra of maximal class.
We refer the reader to the last part
of~\cite[Section~3]{AviMat:A-Z} for a discussion,
we safely exclude characteristic two from our statements, but
provide an interpretation for them in Remarks~\ref{rem:char2-big} and~\ref{rem:char2-prime}.

Resuming our general discussion of Nottingham Lie algebras, it was proved in~\cite{CaMa:Nottingham} that
\begin{equation}
\label{eq:second_chain}
 C_{L_1}(L_{q+1})= \cdots = C_{L_1}(L_{2q-3})=\langle Y \rangle,
\end{equation}
provided $p>5$.
In particular, $L$ cannot have a third diamond in degree lower than $2q-1$.
The assumption on the characteristic was weakened in~\cite[Proposition~5.1]{Young:thesis}
to $p>3$ and $q\neq 5$ (a necessary exception).
We refer the reader to~\cite[Section~2]{AviMat:A-Z} for a
discussion of Nottingham Lie algebras where the third genuine diamond occurs in degree higher than
$2q-1$.
Assuming that $L$ does have a genuine third diamond of finite type $\mu_3$ (hence different from $0$ and $1$,
two exceptional cases which are also discussed in~\cite{CaMa:Nottingham})
in degree $2q-1$, detailed structural information on $L$ was obtained in~\cite{CaMa:Nottingham}, for $p>5$.
 
It was shown there that $L$ is uniquely determined as a thin Lie algebra, that the diamonds occur in each degree
congruent to $1$ modulo $q-1$, and that their types follow an arithmetic progression
(determined by $\mu_2=-1$ and $\mu_3$).
The stated uniqueness of $L$ was proved by showing that the
relations of degree up to $2q$, in the generators $X$ and $Y$,
which $L$ satisfy, actually define a central (or second central in one
case) extension of $L$.
Diamonds appearing at regular intervals, with types following an
arithmetic progression, reveal a periodic
structure which in our context is captured by the following
definition.

\begin{definition}\label{def:loop_algebra}
Let $S$ be a finite-dimensional Lie algebra, over the field $\F$,
with a cyclic grading $S=\bigoplus_{k\in\Z/N\Z}S_k$.
Let $U$ be a subspace of $S_{\bar 1}$ and let $T$ be an indeterminate over $\F$.
The {\em loop algebra} of $S$ with respect to the given grading and the subspace $U$ is the Lie subalgebra of
$S\otimes\F[T]$ generated by $U\otimes T$.
\end{definition}

The subspace $U$ need not be mentioned when it coincides with $S_{\bar 1}$, as will always be the case in this
paper.
Loosely speaking, the loop algebra construction produces an infinite-dimensional Lie algebra from $S$ by {\em replicating} its structure periodically.
Proving that such a loop algebra is thin (in particular, the verification of the covering property~\eqref{eq:covering})
can then conveniently be done inside the finite-dimensional Lie algebra $S$.
Beware that if $S$ is graded over an arbitrary cyclic group $G$ of order $N$ we need
to fix a specific isomorphism of $G$ with $\Z/N\Z$
(or, equivalently, a distinguished generator $\bar 1$ for $G$)
for the corresponding loop algebra to be defined.

Existence results for the Nottingham Lie algebras $L$
whose uniqueness
was proved in~\cite{CaMa:Nottingham} have been proved in various articles,
always by producing the required algebra as a loop algebra of some finite-dimensional Lie algebra $S$
with respect to a suitable cyclic grading.
In each case $L$ has diamonds, possibly fake, in all degrees
congruent to $1$ modulo $q-1$.
If all diamonds are genuine (not fake) then
$S$ is a simple Lie algebra of dimension a power of $p$.
If fake diamonds are present then
$S$ is simple of dimension two less than a power of $p$,
possibly extended by an outer derivation, in the two cases
where backwards continuation of the arithmetic progression of diamond types
would predict the first diamond $L_1$ to be fake
(which it cannot be).

The special case $\mu_3=-1$, where the arithmetic progression of
diamond types is a constant sequence, was dealt with
in~\cite{Car:Nottingham}, taking $S=W(1;n)$, a {\em Zassenhaus algebra}.
When $q=p$ the corresponding loop algebra $L$ is the graded Lie
algebra associated with the lower central series of the Nottingham
group over the field of $p>2$ elements.
For the remaining cases a dichotomy arises according to whether
the arithmetic progression of diamond types is contained in the
prime field, or not, which will entail the absence or presence of
fake diamonds.
The Nottingham algebras with $\mu_3\in\F_p\setminus\{-1\}$
were produced in~\cite{Avi}, as loop algebras of simple {\em graded Hamiltonian algebras} $H(2;(1,n))^{(2)}$
(extended by an outer derivation when $\mu=-2,-3$),
and those with $\mu_3\not\in\F_p$ in~\cite{AviMat:A-Z}, as loop algebras of {\em Albert-Zassenhaus algebras} $H(2:(1,n);\Phi(1))$.
Those with $\mu_3\in\F_p$ can also be recovered from those with $\mu_3\not\in\F_p$
through a deformation argument, which is explained
in~\cite[Section~7]{AviMat:A-Z}.

The third diamond can also have type $\mu_3=\infty$.
In that case a great variety of diamond patterns are possible,
most of which are not periodic, as shown in~\cite[Theorem~3.3]{Young:thesis}.
One periodic possible pattern has diamonds occurring in all
degrees congruent to $1$ modulo $q-1$, with most of them having
infinite type except for one diamond of finite type every $p^s$ diamonds,
for some positive integer $s$,
and with the finite types following an arithmetic progression.
Thus, a diamond of finite type occurs in each degree congruent to
$1$ modulo $p^s(q-1)$, and the arithmetic progression of finite diamond types is determined by the
earliest such diamond after $L_q$, which occurs in degree
$q+p^s(q-1)$.
One such algebra was constructed in~\cite[Section~5]{AviMat:A-Z},
as a loop algebra of an Albert-Zassenhaus algebra
$H(2:(s+1,n);\Phi(1))$, where all finite types equal $-1$.

The main goal of this paper, in Sections~\ref{sec:big-field} and~\ref{sec:prime-field}, is to provide analogous constructions
for the non-constant arithmetic progressions of finite diamond
types.
Here, too, we have a dichotomy according to whether the
progression is entirely contained in the prime field $\F_p$, or
not.
The loop algebra construction will then employ different finite-dimensional Lie algebras
$S$, namely, a simple graded Hamiltonian algebra $H(2;(s+1,n))^{(2)}$
in the former case, and an Albert-Zassenhaus algebra
$H(2:(s+1,n);\Phi(1))$ in the latter.

\section{Grading switching}\label{sec:Laguerre}

The loop algebra construction described in the Section~\ref{sec:nott}
requires specifying a cyclic grading of a finite-dimensional
Lie algebra $S$.
As briefly mentioned there, and more extensively discussed
in~\cite{CaMa:Hamiltonian,AviMat:A-Z}, the Lie algebras $S$ which have been used to
construct thin Lie algebras in this way belong to the classical types $A_1$ and $A_2$
(which are the only ones occurring in characteristic zero as well, see~\cite{CMNS}),
or to the generalized Cartan types $W$ and $H$.

In this section we discuss
a partly new technique to produce suitable cyclic gradings.

Gradings over a cyclic group can be obtained as specializations of
gradings over abelian groups of larger rank, by passing to cyclic quotients of the
latter.
One occurrence of gradings for our Lie algebras $S$ is as generalized root space
decompositions with respect to tori in $\Der(S)$, but Lie algebras
of Cartan type come equipped with natural gradings which are not of that
type (because the grading groups do not have exponent $p$).
Suitable cyclic specializations of the latter gradings can
sometimes produce thin Lie algebras as the corresponding loop
algebras (an example from~\cite{AviMat:A-Z} is recalled at the end of this section),
but in order to produce thin Lie algebras with different diamond
type patterns one may have to pass to new cyclic gradings.

In case of gradings associated with tori, one fundamental tool to {\em switch gradings} is
available.
The {\em toral switching} technique plays a crucial role in the
classification theory of simple modular Lie algebras.
The basic idea, originating from~\cite{Win:toral}, but then substantially generalized
in~\cite{BlWil:rank-two}, and finally~\cite{Premet:Cartan}, is to apply to
a maximal torus of a simple restricted modular Lie algebra some sort of exponential of
an inner derivation, in order to produce a new torus with better properties.
Associated with the new torus is a new grading of the algebra.
Part of the process can be placed in a more general setting, which we describe now, where
a new grading can be obtained from a given one without the intervention of any torus.
To this purpose it is not restrictive to consider only gradings over a cyclic group.

In characteristic zero exponentials of derivations are automorphisms
when they are defined (for example when the derivation is
nilpotent).
This is not the case in positive characteristic,
and toral switching takes advantage of this fact.
In prime characteristic $p$, evaluating the exponential series $\exp(X)=\sum_{i=0}^{\infty}X^i/i!$ on a derivation $D$
only makes sense under the condition $D^p=0$, allowing one to neglect
all terms of the exponential series whose denominator is a multiple of $p$.
More rigorously, instead of the ordinary exponential series
one considers the {\em truncated exponential} $E(X)=\sum_{i=0}^{p-1}X^i/i!$,
which can be evaluated on an arbitrary derivation $D$
but is most closely assimilated to $\exp(D)$ when $D^p=0$.
Although this condition on $D$ does not ensure that $E(D)$ is an automorphism
(only $D^{(p+1)/2}=0$ does, for $p$ odd),
it forces $E(D)$ to send a grading into a grading
in the following result, which is~\cite[Theorem~2.3]{Mat:Artin-Hasse}.

\begin{theorem}\label{thm:exp}
Let $A$ be a non-associative algebra over a field of prime
characteristic $p$, graded over the integers modulo $m$. Suppose
that $A$ has  derivation $D$ such that $D^{p}=0$ graded of degree $d$, with
$m\mid pd$.
Then the direct sum decomposition $A=\bigoplus_{i}E(D)A_{i}$ is a grading over
the integers modulo $m$.
\end{theorem}

This is the {\em grading switching} mentioned earlier.
The main result of~\cite{Mat:Artin-Hasse} is that Theorem~\ref{thm:exp} remains true
under the weaker hypothesis that $D$ is nilpotent, provided that $E(D)$ is replaced by
$E_p(D)$, where $E_p(X)$ is the {\em Artin-Hasse exponential} series.

The nilpotency assumption on $D$ is too restrictive for the applications in this paper.
In~\cite[Theorem~5.1]{AviMat:Laguerre} we have further extended Theorem~\ref{thm:exp} to arbitrary derivations,
where the place of $E_p(D)$ is taken by certain maps defined by means of
(generalized) Laguerre polynomials.
To avoid unnecessary complications we limit ourselves to quoting the special case we need
of that general result, where $D$ satisfies a certain equation.

The classical (generalized) Laguerre polynomial of degree
$n \geq 0$ can be defined as
$L_n^{(\alpha)}(X)=\sum_{k=0}^{n}\binom{\alpha+n}{n-k}(-X)^k/k!$.
The parameter $\alpha$ is classically a complex number, but the definition allows one to take $\alpha$ in any field
where $n!$ is invertible.
Let $\F$ be a field of characteristic $p>0$, let $\alpha \in \F$, and consider the Laguerre polynomial
\[
L_{p-1}^{(\alpha)}(x)=\sum_{k=0}^{p-1}\binom{\alpha+p-1}{p-1-k} \frac{(-x)^k}{k!}.
\]
When $\alpha=0$ this coincides with
the truncated exponential $E(X)$.

\begin{theorem}\label{thm:special}
Let $A=\bigoplus_k A_k$ be a non-associative algebra over a field $\F$ of prime characteristic $p$,
graded over the integers modulo $m$.
Suppose that $A$ has a graded derivation $D$ of degree $d$
such that  $D^{p^2}=\lambda^{(p-1)p}D^p$ for some nonzero $\lambda \in \F$, with $m\mid pd$.
Suppose that there exists $\pi \in \F$  with
$\pi^p-\pi=\lambda^p$.
Let $A=\bigoplus_{a\in\F_p}A^{(a\lambda)}$ be the decomposition of $A$
into a direct sum of generalized eigenspaces for $D$,
and let $\mathcal{L}_D:A\to A$ be the linear map whose restriction to
$A^{(a\lambda)}$coincides with $L_{p-1}^{(a\pi)}(D)$.
Then
$A=\bigoplus_k \mathcal{L}_D(A_k)$ is a grading of $A$ over the integers modulo $m$.
\end{theorem}

\begin{proof}
The conclusion follows from~\cite[Theorem~5.1]{AviMat:Laguerre} by taking $h(T)=0$ and
$g(T)=\pi\lambda^{-p} T^p$, whence $g(a \lambda)=a \pi$.
\end{proof}

In the special case where $D^p=0$, the value of $\lambda$ is immaterial because $A=A^{(0)}$,
and Theorem~\ref{thm:special} becomes Theorem~\ref{thm:exp}.

\section{Certain Lie algebras of Cartan type}\label{sec:Cartan}

We briefly recall here the definitions
of certain Lie algebras of Cartan type belonging to the Hamiltonian series.
A broader discussion of these Lie algebras in the context of thin Lie algebras
can be found in~\cite{AviMat:A-Z}.

Let $\F$ be a field of prime characteristic $p$.
The algebra $\mathcal{O}(1,n)$ of \emph{divided powers} in one indeterminate $x$ of height $n$,
is the associative $\F$-algebra with basis elements $x^{(i)}$, for $0\leq i\leq p^{n}$,
and multiplication defined by
$x^{(i)}\cdot x^{(j)}=\binom{i+j}{i} x^{(i+j)}$.
The algebra
$\mathcal{O}(2,(n_1, n_2))$ of divided powers in two
indeterminates $x$ and $y$ of heights $n_1$ and $n_2$
may be identified with the tensor product algebra
$\mathcal{O}(1;n_1)\otimes\mathcal{O}(1;n_2)$.
A basis is given by the monomials
$x^{(i)}y^{(j)}$, for
$0\leq i<p^{n_1}$ and $0\leq j<p^{n_2}$, which are multiplied according to the rule
$x^{(i)} y^{(j)}x^{(k)}y^{(l)}=\binom{i+k}{i}\binom{j+l}{j} x^{(i+k)}y^{(j+l)}$.
We use the
standard shorthands $\bar{x}=x^{(p^{n_{1}}-1)}$ and
$\bar{y}=y^{(p^{n_{2}}-1)}$.
In the algebra $\mathcal{O}(1,1)$, which is isomorphic with $\F[x]/(x^p)$, we
may define a sort of generalized power
$(1+x)^{\alpha} = \sum_{i=0}^{p-1} \binom{\alpha}{i}i! x^{(i)}$,
where the exponent $\alpha$ is an arbitrary element of $\F$.
When $\alpha=0,1,\cdots, p-1$ this expression specializes to
the usual binomial theorem expressed in terms of divided powers.

The simple graded Hamiltonian algebra $H(2;(n_1,n_2))^{(2)}$
is defined as a subalgebra of the algebra of derivations of $\mathcal{O}(2,(n_1,
n_2))$.
However, for the present purposes we can identify it with the subspace
of $\mathcal{O}(2;(n_1,n_2))$ spanned by its monomials $x^{(i)}y^{(j)}$
with $(i,j)\neq(0,0),(p^{n_1},p^{n_2})$,
endowed with the Lie bracket
\begin{equation}\label{eq:Poisson_0}
\begin{split}
\{x^{(i)}
y^{(j)},
x^{(k)}
y^{(l)}\}
&=
x^{(i)}
y^{(j-1)}
x^{(k-1)}
y^{(l)}-
x^{(i-1)}
y^{(j)}
x^{(k)}
y^{(l-1)}
\\
&=
N(i,j,k,l)\,
x^{(i+k-1)}
y^{(j+l-1)},
\end{split}
\end{equation}
where 
\[
N(i,j,k,l):=
\binom{i+k-1}{i}
\binom{j+l-1}{j-1}-
\binom{i+k-1}{i-1}
\binom{j+l-1}{j}.
\]
This Lie algebra is clearly graded (hence its name) over
$\Z\times\Z$, by assigning degree $(i,j)$ to the monomial
$x^{(i+1)}y^{(j+1)}$.
In odd characteristic it is a simple Lie algebra, of dimension $p^{n_1+n_2}-2$,
In characteristic two it is simple only when $n_1>1$ and $n_2>1$.

The Albert-Zassenhaus algebra
$H(2:(n_1,n_2);\Phi(1))$ can be identified with
$\mathcal{O}(2;(n_1,n_2))$ with the Lie bracket
\begin{align}
&\{x^{(i)}
y^{(j)},
x^{(k)}
y^{(l)}\}
=
N(i,j,k,l)\,
x^{(i+k-1)}
y^{(j+l-1)}
\quad\text{if $i+k>0$, and}\label{eq:Poisson_1}
\\
&\{y^{(j)},
y^{(l)}\}
=
\left(
\binom{j+l-1}{l}-
\binom{j+l-1}{j}
\right)\bar x
y^{(j+l-1)}.\label{eq:Poisson_exception}
\end{align}
Equations~\eqref{eq:Poisson_1} and~\eqref{eq:Poisson_exception} show that
$H(2;(n_1,n_2);\Phi(1))$
is graded over the group $\Z/p^{n_1}\Z\times\Z$
by assigning degree $(i+p^{n_1}\Z,j)$ to the monomial $x^{(i+1)}y^{(j+1)}$.
In odd characteristic this Lie algebra is simple.
In characteristic two its derived subalgebra $H(2;(n_1,n_2);\Phi(1))^{(1)}$ is simple,
and is spanned by the monomials $x^{(i)}y^{(j)}$
with $(i,j)\neq(p^{n_1},p^{n_2})$.

A {\em specialization} of this grading of
$H=H(2;(n_1,n_2);\Phi(1))$
was considered in~\cite[Section~5]{AviMat:A-Z},
namely, a cyclic grading
$H=\bigoplus_{k \in \Z/N \Z}H_k$
obtained by assigning degree $(1-q)i-j+N\Z$ to the
monomial $x^{(i+1)}y^{(j+1)}$,
where $q=p^{n_2}$ and $N=p^{n_1}(q-1)$.
The monomials $x$ and $\bar y$ acquire degree $1$ in this grading,
and one easily sees that they generate $H$.
According to~\cite[Theorem~5.1]{AviMat:A-Z}, if $p>2$
the loop algebra of $H$ with respect to this grading
is a Nottingham algebra with diamonds in all degrees
congruent to one modulo $q-1$,
all of type $\infty$
except for those in degree congruent to $q$ modulo $p^{n_1}(q-1)$,
which have type $-1$.
In the next two sections we show that different gradings of $H=H(2;(n_1,n_2);\Phi(1))$
lead to similar Nottingham Lie algebras, with the diamonds of type $-1$
replaced by diamonds whose types follow more general arithmetic
progressions.

\section{The big field case}\label{sec:big-field}

This section and the next one contain our main results, which are explicit constructions for the new
Nottingham Lie algebras announced in the Introduction.
They have a diamond in each degree congruent to $1$ modulo $q-1$,
with diamonds having infinite type except for one of finite type every
$p^s$ diamonds, and the finite types follow an arithmetic progression.
In this section we consider the case where the arithmetic progression is not entirely contained in the prime field.
A minimal requirement on the base field $\F$ is that it should contain the
the various finite diamond types, but we allow ourselves to further enlarge it later as needed.

We make the assumption that $p$ is odd and postpone a discussion of the case $p=2$ to Remark~\ref{rem:char2-big}.
Consider the Albert-Zassenhaus algebra $H=H(2;(s+1,n);\Phi(1))$,
for some $s,n>0$, and its derivation $D=(\ad y)^{p^s}$.
Writing each monomial in
$\mathcal{O}(2;(s+1,n))$
in the form  $x^{(ap^s)} x^{(k+1)}y^{(j+1)}$,
with $0\le a<p$, $-1\leq k<p^s-1$ and $-1\leq j<p^n-1$,
we have
\[
D(x^{(ap^s)} x^{(k+1)}y^{(j+1)})=
\begin{cases}
x^{((a-1)p^s)}x^{(k+1)}y^{(j+1)}& \textrm{if $a>0$,}\\
-jx^{((p-1)p^s)}x^{(k+1)}y^{(j+1)}& \textrm{if $a=0$.}
\end{cases}
\]
Consequently, $D^p$ acts semisimply on $H$, with eigenvalues in the prime field, as
\[
D^{p}(x^{(ap^s)} x^{(k+1)}y^{(j+1)})=-jx^{(ap^s)}
x^{(k+1)}y^{(j+1)},
\]
and hence $D^{p^2}=D^p$.

Consider the cyclic grading
$H=\bigoplus_{\ell\in \Z/N \Z}H_\ell$
obtained by assigning degree $(1-q)(ap^s+k)-j+N\Z$ to the
monomial $x^{(ap^s)}x^{(k+1)}y^{(j+1)}$,
where $q=p^n$ and $N=p^{s+1}(q-1)$.
This is the grading from~\cite[Section~5.1]{AviMat:A-Z}
which we recalled at the end of Section~\ref{sec:Cartan}, just in
a different notation.
Thus, $x$ and $\bar y$ acquire degree one in this grading.
The derivation $D$ is graded of degree $N/p$.
Furthermore, each homogeneous element in the grading is an
eigenvector for $D^p$.
Let $\sigma,\pi\in\F$ be such that $(\pi^p-\pi)\sigma^p=1$.
We apply the grading switching from~\cite{AviMat:Laguerre} to this grading,
in the form recalled earlier as Theorem~\ref{thm:special}
but with the derivation $\sigma^{-1}D$ in place of $D$,
and thus obtain another cyclic grading of $H$.

Explicitly, consider the elements
\[
\bar{e}_{j,k,a}=L_{p-1}^{(-j \pi)}\left(\sigma^{-1}D\right)
 (x^{(ap^{s})}x^{(k+1)}y^{(j+1)}),
\]
for $-1\le j<q-1$, $-1\le k<p^{s}-1$, and $a \in \F_{p}$.
According to Theorem~\ref{thm:special} they constitute a  basis of $H$,
graded over the integers modulo $N$ by assigning
$\bar{e}_{j,k,a}$ degree $(1-q)(ap^s+k)-j+N\Z$.
Each homogeneous component in this grading has dimension one, except for those in
degrees congruent to $1$ modulo $q-1$, which are two-dimensional.
For later convenience we multiply $\bar{e}_{j,k,a}$ by the
scalar
$c_{j,a}=a!\sigma^{a}\tbinom{-j \pi+a}{a}\tbinom{ -j\pi+p-1}{p-1}^{-1}$,
obtaining the elements
\begin{equation*}
\displaystyle{e_{j, k, a}=
c_{j,a}\bar{e}_{j,k,a}=
(1+ \sigma x^{(p^{s})})^{-j\pi +a}x^{(k+1)}y^{(j+1)}},
\end{equation*}
for $-1\le j<q-1$, $-1\le k<p^{s}-1$, and $a \in \F_{p}$.

Products between the basis elements $e_{j,k,a}$ are easily computed, and one finds
\begin{equation}\label{eq:eeq1}
\displaystyle{\{e_{j, k, a}, e_{l, h, b}\}=\left( \tbinom{k+h+1}{h} \tbinom{j+l+1}{j}-
\tbinom{k+h+1}{k} \tbinom{j+l+1}{l}\right)e_{j+l,k+h,a+b}}
\end{equation}
for $k+h>-1$ and
\begin{equation}\label{eq:eeq2}
\displaystyle{\{e_{j,-1,a}, e_{l,-1,b}\}=\sigma \left( \tbinom{j+l+1}{j}(-l \pi+b)-
\tbinom{j+l+1}{l}(-j \pi+a)\right)e_{j+l,p^s-2,a+b-1}}.
\end{equation}
In fact, because the basis is graded according to Theorem~\ref{thm:special}, the Lie bracket
$\{e_{j,k,a}, e_{l,h,b}\}$ is a scalar multiple of $e_{j+l,k+h,a+b}$, for $k+h>-2$.
Thus, we only have to compute the scalar factor, for example by computing
the coefficient of $x^{(k+h+1)}y^{(j+l+1)}$ in the result, noting that
$x^{(k+1)}y^{(j+1)}$ always appears with coefficient $1$ in $e_{j,k, a}$.
To do this it suffices to compute the Lie bracket of the
only relevant terms, which is $\{x^{(k+1)}y^{(j+1)},x^{(h+1)}y^{(l+1)}\}$.
Similarly, in the case $k+h=-2$, the
Lie bracket $\{e_{j, -1, a}, e_{l, -1, b}\}$ is a scalar multiple
of $e_{j+l,p^s-2,a+b-1}$.
The scalar can be recovered by computing
the coefficient of $x^{p^s-1}y^{j+l+1}$ in the result.
Here the Lie bracket of the relevant terms is
\[
\sigma (-j \pi +a)\{x^{(p^s)}y^{(j+1)}, y^{(l+1)}\}+
 \sigma (-l \pi +b)\{y^{(j+1)}, x^{(p^s)}y^{(l+1)}\}.
\]

We are now ready to state and prove the main result of this section.

\begin{theorem}\label{thm:big_field}
Let $\F$ be a field of odd characteristic $p$, let $n,s$ be a positive integer, and set $q=p^n\ge p$.
Assume that there are $\sigma,\pi\in\F$ with $(\pi^p-\pi)\sigma^p=1$.
Then the elements
\[
e_{j,k,a}=(1+ \sigma x^{(p^{s})})^{-j\pi +a} x^{(k+1)}y^{(j+1)},
\]
for $-1\le j<q-1$, $-1\le k<p^{s}-1$, and $a \in \F_{p}$,
form a graded basis of $H(2;(s+1,n);\Phi(1))$ over the integers modulo
$(q-1)p^{s+1}$, where
$e_{j,k, a}$ has degree
$(1-q)(a p^{s}+k)-j \pmod{(q-1)p^{s+1}}$.

The corresponding loop algebra $L$ is thin, with second diamond in degree $q$.
The diamonds occur in each degree
of the form $t(q-1)+1$, with the diamond type being finite exactly when $t\equiv 1\mod p^{s}$.
Those finite types follow
an arithmetic progression, not entirely contained in the prime field.
The diamond in degree $(q-1)(p^s+1)+1$ has type
$\nu=-1+1/\pi$.
\end{theorem}

Note that the type $\nu$ of the diamond in degree
$(q-1)(p^s+1)+1$, with $\nu\in \F\setminus\F_{p}$, determines $\sigma$ and $\pi$ uniquely
via
\[
\pi^p-\pi=1/\sigma^p, \quad \nu=-1+1/\pi.
\]

\begin{proof}
We have already deduced from Theorem~\ref{thm:special}
that the elements $e_{j,k,a}$ form a graded basis of $H$
with respect to a cyclic grading, with degrees as specified.
In particular, the elements
\[
X= e_{-1, 0, 0}=(1+ \sigma x^{(p^{s})})^{\pi} x
\qquad\textrm{and}\qquad
Y= e_{q-2, -1, 0}=(1+ \sigma x^{(p^{s})})^{2 \pi} \bar{y}
\]
span the homogeneous component of degree $1$ in this grading.
Form the tensor product of $H$ with a polynomial ring $\F[T]$.
According to Definition~\ref{def:loop_algebra} the loop algebra $L$ is the subalgebra of $H\otimes \F[T]$ generated by the elements
$X\otimes T$ and $Y\otimes T$.
We need to prove that the covering property holds in $L$, and that the diamonds have the types specified.
However, it is possible and notationally more convenient to prove those conclusions working inside
the Hamiltonian algebra $H$
rather than in its loop algebra $L$.

We have already given the full multiplication table of the elements $e_{j,k,a}$ in Equations~\eqref{eq:eeq1} and~\eqref{eq:eeq2}.
In particular, for $j \neq -1,0$ we read off that
$\{e_{j,k,a},Y\}=0$ and $\{e_{j,k,a},X\}=e_{j-1,k,a}$.
Thus, all one-dimensional homogeneous components which do not immediately precede a two-dimensional component are centralized
by $Y$, and satisfy the covering property.
By this expression we mean that each of them covers the next component,
in the sense that for each nonzero element $u$ in any of them, $\{u,L_1\}$ equals the following component.
We deal with the remaining components next.

Each element $v=e_{0,-1, a}=(1+\sigma x^{(p^{s})})^{a} y$,
which has degree $(1-q)(ap^{s}-1)$,
spans the homogeneous component just preceding a diamond of finite type.
In fact, we have
 \begin{align*}
 &\{v,X\}=e_{-1,-1, a}\\
 &\{v,Y\}=\sigma(2 \pi + a)e_{q-2,p^s-2,a-1}\\
 &\{v,X,X\}=0=\{v,Y,Y \}\\
 &\{v,X,Y\}=-\sigma(\pi + a )e_{q-3,p^s-2,a-1}\\
 &\{v,Y, X\}=\sigma (2 \pi + a)e_{q-3,p^s-2, a-1}
 \end{align*}
Besides proving that the covering property holds in the relevant components, these equations show that
$\mu\{v,Y,X\}=(1- \mu)\{v,X,Y\}$ where $\mu= -a \nu -a-1$.
According to Equation~\eqref{eq:diamond_type}
we are in the presence of a diamond of type $\mu$.
In particular, the element $e_{0,-1,0}$, of degree $q-1$, immediately precedes
the second diamond, of type $-1$ as always,
and the element $e_{0,-1,-1}$, of degree $(q-1)(p^s+1)$, immediately precedes a diamond of type $\nu$.
These two diamond types determine an arithmetic progression, which describes all the diamond types considered here.

Finally, we show that each element $w=e_{0,k,a}=(1+\sigma x^{(p^{s})})^{a} x^{(k+1)}y$, for $0\le k<p^{s}-1$,
which has degree $(1-q)(ap^{s}+k)$,
spans the homogeneous components which immediately precedes a diamond of infinite type.
In fact, we have
 \begin{align*}
 &\{w,X\}=e_{-1,k,a}\\
 &\{w,Y\}=-e_{q-2,k-1,a}\\
 &\{w,X,X\}=0=\{w,Y,Y\}\\
 &\{w,X,Y\}=-e_{q-3,k-1,a}\\
 &\{w,Y,X\}=e_{q-3,k-1,a}=-\{w,X,Y\}
 \end{align*}
We have proved that the loop algebra $L$ is thin, and that the diamonds types are as specified.
\end{proof}

\begin{rem}\label{rem:char2-big}
Theorem~\ref{thm:big_field} admits an interpretation in the forbidden case of characteristic two.
As mentioned in Section~\ref{sec:Cartan},
the algebra $H=H(2;(s+1,n);\Phi(1))$ is not simple in characteristic two, but its derived subalgebra $H^{(1)}$ is.
It is spanned by the same monomials except $\bar x\bar y$, and hence has dimension $2^{s+n+1}-1$.
All our basis elements $e_{j,k,a}$ belong to $H^{(1)}$ except for
$e_{q-2,2^s-2,1}=x^{(p^s-1)}\bar{y}+\sigma \bar{x}\bar{y}$.
This is one of the two basis elements of degree $q$ modulo $(q-1)2^{s+1}$,
which usually span the second diamond in the loop algebra.
The elements $X=e_{-1,0,0}$ and $Y=e_{q-2,-1,0}$ generate $H^{(1)}$.
The corresponding loop algebra is still thin,
but the second diamond becomes fake of type $1\equiv -1 \bmod 2$.
In fact, the homogeneous component of degree $q-1$ is generated by $v=e_{0,-1,0}$, and we have
$\{v,Y\}=2 \pi \sigma e_{q-2,2^s-2,1}= 0$.
\end{rem}

\begin{rem}\label{rem:s=0}
When $s=0$ the loop algebra of $H(2;(1,n);\Phi(1))$ according to the grading given
in Theorem~\ref{thm:big_field} is a Nottingham algebra with all diamonds of finite type.
This algebra was constructed in~\cite{AviMat:A-Z} as a loop algebra of $H(2;(1,n);\Phi(1))$
with respect to a certain grading, which, however, differs from the grading given here.
A very similar discrepancy occurred in~\cite{AviMat:-1}, and the detailed explanation
for it given in~\cite[Remark~4.3]{AviMat:-1} applies here, almost verbatim.
\end{rem}

\section{The prime field case}\label{sec:prime-field}

In this final section we construct Nottingham Lie algebras with $p^{s}-1$ diamonds of infinite type
separated by single occurrences of a diamond of finite type, with the finite types forming a nonconstant arithmetic progression in the prime field.
Assume $p$ odd and let $H=H(2;(s+1,n))^{(2)}$, for some $s>0$, which has dimension $p^{s+n+1}-2$.
Set $q=p^{n}$ and $N=p^{s+1}(q-1)$.
The derivation
$D=(\ad y)^{p^s}$ of $H$ satisfies $D^p=0$, because
\[
D(x^{(ap^s)} x^{(k+1)}y^{(j+1)})=x^{((a-1)p^s)}x^{(k+1)}y^{(j+1)}.
\]
Thus, in this case the exponential of $D$ makes sense on the whole algebra $H$.
Let $\pi \in \F_{p}$ with $\pi\neq 0$.
We obtain a grading of $H$ over
$\Z/N\Z$ by assigning the monomial $x^{(ap^s)}x^{(k+1)}y^{(j+1)}$ degree
$(1-q)\bigl((a+j \pi)p^s+k\bigr)-j+N\Z$.
Because the derivation $D$ is graded of degree $N/p$ we can apply Theorem~\ref{thm:exp}
and obtain another grading of $H$ over the integers modulo $N$.
A new graded basis, with degrees given by the same formula above, consists of the elements
\[
e_{j,k,a}=a! E(D)(x^{(ap^s)}x^{(k+1)}y^{(j+1)})=(1+x^{(p^s)})^{a}\, x^{(k+1)}y^{(j+1)},
\]
where $0\le a<p$, $-1 \leq k <p^s-1$ and $-1 \leq j < q-1$, with
$(j,k,a)\not=(-1,-1,0),(q-2,p^{s}-2,p-1)$.
For later convenience we  set
$e_{-1,-1,0}=0=e_{q-2,p^s-2,p-1}$.
As in the previous section the products of the elements $e_{j,k,a}$ are easily obtained, and are found to be
\begin{equation}\label{eq1}
\{e_{j,k,a},e_{l,h,b}\}= \left( \tbinom{k+h+1}{h} \tbinom{j+l+1}{j}-
 \tbinom{k+h+1}{k} \tbinom{j+l+1}{l}\right)e_{j+l,k+h, a+ b}
\end{equation}
for $k+h>-2$, and
\begin{equation}\label{eq2}
\{e_{j,-1,a},e_{l,-1,b}\}=\left(b \tbinom{j+l+1}{j}- a \tbinom{j+l+1}{l}\right)e_{j+l,p^s-2,a+ b-1}.
\end{equation}

\begin{theorem}\label{thm:prime_field}
Let $\F$ be a field of odd characteristic $p$, let $q=p^n\ge p$, and let $\pi$ be a nonzero element of $\F_p$.
Then the elements
\[
e_{j,k,a}=(1+x^{(p^s)})^a\, x^{(k+1)}y^{(j+1)},
\]
for $0\le a<p$, $-1 \leq k <p^s-1$ and $-1 \leq j < q-1$,
form a graded basis of $H(2;(s+1,n))^{(2)}$ over the integers modulo $(q-1)p^{s+1}$,
where $e_{j,k, a}$ has degree
$(1-q)\bigl((a+j \pi)p^s+k\bigr)-j\pmod{(q-1)p^{s+1}}$.

The corresponding loop algebra $L$ is thin, with second diamond in degree $q$.
The diamonds occur in all degrees of the form $t(q-1)+1$,
with the diamond type being finite exactly when $t \equiv 1 \mod p^{s}$.
Those finite types follow an arithmetic progression contained in the prime field.
The diamond in degree $(q-1)(p^s+1)+1$ has type $\nu=-1+1/\pi$.
\end{theorem}

\begin{proof}
The preceding discussion shows that the elements $e_{j,k,a}$ form a graded basis of $H$
with respect to a cyclic grading, with degrees as specified.
In particular, the elements
\[
X=e_{-1,0,\pi}=(1+x^{(p^{s})})^{\pi}x
\qquad\textrm{and}\qquad
Y=e_{q-2,-1,2\pi}=(1+x^{(p^{s})})^{2\pi}\bar y
\]
span the homogeneous component of degree $1$ in this grading.
As in the proof of Theorem~\ref{thm:big_field}, we conveniently check the covering property and the diamond types in $H$ rather than in its loop algebra $L$.

For $j \neq -1,0$ we have
$\{e_{j,k,a},X\}=e_{j-1,k,a+\pi}$ and
$\{e_{j,k,a}, Y\}=0$.
Next, each element $v=e_{0,-1,a}$, of degree $(1-q)(ap^{s}-1)$,
spans the homogeneous component preceding a diamond of finite type $\mu= -a \nu -a-1$, because
\begin{align*}
&\{v,X\}=e_{-1,-1,a+\pi}\\
&\{v,Y\}=(2\pi+a)e_{q-2,p^s-2,a-1+2\pi}\\
&\{v,X,X\}=0=\{v,Y,Y\}\\
&\{v,X,Y\}=-(\pi+a) e_{q-3,p^s-2,a-1+3\pi}\\
&\{v,Y,X\}=(2\pi+a)e_{q-3,p^s-2,a-1+3\pi}
\end{align*}
As in the proof of Theorem~\ref{thm:big_field} these diamond types follow the stated arithmetic progression.
Finally, each element $w=e_{0,k,a}$, for $0\le k<p^s-1$, of degree $(1-q)(ap^s+k)$,
occurs just before a diamond of infinite type, because
\begin{align*}
&\{w,X\}=e_{-1,k, a+\pi}\\
&\{w,Y\}=e_{q-2,k-1,a+2\pi}\\
&\{w,X,X\}=0=\{w,Y,Y\}\\
&\{w,X,Y\}=-e_{q-3,k-1,a+3\pi}\\
&\{w,Y,X\}=e_{q-3,k-1,a+3\pi}=-\{w,X,Y\}
\end{align*}
The proof is complete.
\end{proof}

The case $\pi=0$, which is excluded above, should correspond to taking $\nu =\infty$.
However, the grading in Theorem~\ref{thm:prime_field} does not produce a thin Lie algebra because $\{e_{0,-1,0},X,Y\}=0=\{e_{0,-1,0},Y,X\}$.

\begin{rem}\label{rem:char2-prime}
Differently from Theorem~\ref{thm:big_field}, which needs to be modified in characteristic two
as described in Remark~\ref{rem:char2-big}, Theorem~\ref{thm:prime_field}
remains true as stated when $p=2$, as long as $q>2$.
However, except for a brief mention in Section~\ref{sec:nott}
we have not introduced thin Lie algebras of characteristic two in this paper.
The discussion in~\cite[Section~3]{AviMat:A-Z} explains how the relevant terminology
is to be interpreted, including the peculiar fact that the second diamond is fake.
In fact, in the case under discussion $\pi=1$, and the diamonds of finite type of the loop algebra $L$ are all fake, of alternate types $0$ and $1$.
\end{rem}

\bibliography{References}
\end{document}